\documentclass[a4paper,12pt]{amsart}
\usepackage{amsmath,amsthm,amssymb}
\usepackage[arrow,matrix]{xy}
\usepackage{hyperref}

\theoremstyle{ams}
\newtheorem{theorem}{Theorem}[section]
\newtheorem{proposition}[theorem]{Proposition}
\newtheorem{lemma}[theorem]{Lemma}
\newtheorem{corollary}[theorem]{Corollary}
\numberwithin{equation}{section}
\theoremstyle{definition}
\newtheorem{definition}[theorem]{Definition}

\newcommand{\C}{\mathbb{C}}
\newcommand{\Q}{\mathbb{Q}}

\newcommand{\Z}{\mathbb{Z}}
\newcommand{\CP}{\mathbb{C}P}
\newcommand{\mathscr}{}

\newcommand{\ma}{\alpha}

\DeclareMathOperator{\rank}{rank}

\def\c#1#2{c_{#2#1}} 
\def\d#1#2{d_{#2#1}} 
\def\a#1#2{a_{#2#1}} 

\begin{document}
\title {Properties of Bott manifolds and Cohomological Rigidity}

\author {Suyoung Choi}
\address{Department of Mathematics, Osaka City University, Sugimoto, Sumiyoshi-ku, Osaka 558-8585, Japan}
\email{choi@sci.osaka-cu.ac.jp}
\urladdr{http://math01.sci.osaka-cu.ac.jp/~choi}

\author{Dong Youp Suh}
\address{Department of Mathematical Sciences, KAIST, 335 Gwahangno, Yu-sung Gu, Daejeon 305-701, Korea}
\email{dysuh@math.kaist.ac.kr}

\thanks{The first author is supported by the Japanese Society for the Promotion of Sciences (JSPS grant no. P09023).}
\thanks{The second author is partially supported by Basic Science Research Program through the national Research Foundation of Korea(NRF) founded by the Ministry of Education, Science and Technology (2009-0063179).}

\subjclass[2000]{Primary 57S25; Secondary 22F30}

\keywords{Toric manifold, Quasitoric manifold, Bott tower, twist number, cohomological complexity, cohomological rigidity, one-twisted Bott tower}

\date{\today}
\maketitle
\begin{abstract}
The cohomological rigidity problem for toric manifolds asks whether the cohomology ring of a toric manifold determines the topological type of the manifold. In this paper, we consider the problem with the class of one-twist Bott manifolds to get an affirmative answer to the problem. We also generalize the result to quasitoric manifolds. In doing so, we show that the twist number of a Bott manifold is well-defined and is equal to the cohomological complexity of the cohomology ring of the manifold. We also show that any cohomology Bott manifold is homeomorphic to a Bott manifold. All these results are also generalized to the case with $\Z_{(2)}$-coefficients, where $\Z_{(2)}$ is the localized ring at $2$.
\end{abstract}

\tableofcontents

\section{Introduction}
A class $\mathcal M$ of closed manifolds is said to be {\em cohomologically rigid} if any two elements $M,N\in \mathcal M$ are homeomorphic whenever their cohomology rings are isomorphic.
One of the intersting problems in toric topology is to determine whether the class of toric (or quasitoric) manifolds is cohomologically rigid. A quasitoric manifold is a topological
analogue of a toric manifold, which was first introduced by Davis and Januszkiewicz in \cite{DJ}, see also \cite{BP}.

Since the class of toric or quasitoric manifolds is too large to handle it is reasonable to restrict our attention to a smaller but an interesting subclass of manifolds. Namely, we would like to restrict our focus on Bott manifolds or cohomology Bott manifolds.

A \emph{(complex) Bott tower} $\{B_j \mid j=0,\ldots, n\}$ of height $n$ (or \emph{$n$-stage Bott tower}) is a sequence
$$
B_n\stackrel{\pi_n}\longrightarrow B_{n-1} \stackrel{\pi_{n-1}}\longrightarrow
\dots \stackrel{\pi_2}\longrightarrow B_1 \stackrel{\pi_1}\longrightarrow
B_0=\{\text{a point}\},
$$ of manifolds  $B_j = P( \underline{\C} \oplus \xi_{j-1} )$
where $\xi_{j-1}$ is a complex line bundle over $B_{j-1}$ for each
$j=1,\ldots,n$. In this case we call $B_j$ the {\em $j$-th stage Bott manifold} of the Bott tower.
A smooth manifold $M$ diffeomorphic to the top stage $B_n$ of a Bott tower is also called a Bott manifold, and in this case
$\{B_j\mid j=0,\ldots n\}$ is callded a {\em Bott tower structure } of $M$.

A Bott tower was first introduced
by Bott and Samelson in \cite{BS}, and later named as Bott tower in \cite{GK}.
Bott manifolds are known to have algebraic torus actions, hence they constitute
an important family of toric manifolds. A \emph{cohomology Bott manifold} is a quasitoric manifold whose cohomology ring is isomorphic to that of a Bott manifold.

The question we are interested in  here is whether the class of (cohomology) Bott manifolds is cohomologically rigid.
So far, there is no counter example to the question, but some positive results.
Masuda and Panov considered the problem and showed that any $n$-stage Bott manifold is diffeomorphic to
the trivial Bott manifold $(\mathbb CP^1)^n$ if its cohomology ring is isomorphic to that of $(\mathbb CP^1)^n$.

The notion of Bott tower is generalized to a {\em generalized Bott tower} in \cite{ch-ma-su08} which is an iterated complex projective space bundles obtained from projectivization of sum of line bundles over a complex projective
space, and the result in \cite{MP} is extended to generalized Bott manifolds in \cite{CMS}.
Furthermore any three-stage Bott manifolds and $2$-stage generalized Bott manifolds  are shown to be cohomologically rigid there.

Davis and Januszkiewicz also introduced a real analogue of a quasitoric manifold called a {\em small cover} in \cite{DJ}.
But for small covers the corresponding cohomologies are with $\mathbb Z_2$-coefficients.
Moreover we can define a {\em real Bott tower} to be an iterated $\mathbb RP^1$ bundles over $\mathbb RP^1$, and a
{\em generalized real Bott tower} is defined similarly.
So one might ask a similar cohomological rigidity question asking whether  two real Bott manifolds are homeomorphic if their
mod $2$ cohomology rings are isomorphic.
This is shown to be true recently by Kamishima and Masuda \cite{KM}, \cite{Ma1}.
However the same question for generalized real Bott manifolds is not true, see \cite{Ma2}.

However not much is known for the cohomological rigidity of  Bott  manifolds whose
cohomology rings are not isomorphic to that of product of $\mathbb CP^1$.
In this article we consider one-twist Bott manifolds, i.e., only one stage has nontrivial fibration
in its Bott tower structure. We prove in Theorem~\ref{theorem:rigidity of one-twist Bott tower}
that one-twist Bott towers are cohomologically rigid.
Moreover this result is extended to quasitoric manifolds whose $\mathbb Z$-cohomology rings are isomorphic to those of one-twist Bott towers in Theorem~\ref{theorem:final theorem}.
Theorem~\ref{theorem:final theorem} is an immediate consequence of Theorem~\ref{theorem:rigidity of one-twist Bott tower}
together with two properties related with Bott towers. They are Theorem~\ref{theorem:twist number = complexity} and
Theorem~\ref{theorem:BQ-algebra and Bott tower}.

A Bott manifold $M$ may have two Bott tower structures $\{B_j \mid j=0,\ldots, n\}$  and
$\{B_j' \mid j=0,\ldots, n\}$. The question we are interested in here is whether the twist number (i.e., the number of
nontrivial fibrations) of the two Bott tower structures are equal? If so, the twist number of a Bott manifold is
well-defined.

On the other hand the cohomology ring of an $n$-stage Bott manifold $M$ is a truncated polynomial ring
\begin{equation}\label{eqn:cohomology}
    H^\ast(M) \cong \Z[x_1, \ldots, x_n]/I,
\end{equation} where $I=<x_j(x_j - f_j) : j=1, \ldots, n>$ and $f_j = \sum_{i=1}^{j-1}\c{j}{i}x_i$ with $\deg x_i=2$.
If the fibration of the $j$-th stage of a Bott tower structure on $M$ is trivial,
then we may assume that $f_j=0$. The number of nonzero $f_j$'s may depend on  the choices of both
generators of the cohomology ring $H^\ast(M)$ and Bott tower structures of $M$.
The {\em cohomological complexity of $M$} is the minimal number of nonzero $f_j$'s
among all possible such choices.
It is obvious that cohomological complexity of $H^\ast(M)$  is less than or equal to the twist number of
any Bott tower structure of $M$.
In Theorem~\ref{theorem:twist number = complexity}, we show that the twist number of any Bott tower structure of
$M$ is equal to the cohomological complexity of $H^\ast(M)$. In particular the twist number of a Bott manifold is
well-defined, namely, it does not depend on the choice of Bott manifold structures of a Bott manifold.

A $BQ$-algebra of rank $n$ is defined in \cite{MP}. In particular, the cohomology ring of any $n$-stage Bott manifold is a $BQ$-algebra of rank $n$ over $\mathbb Z$.
The converse of this is proved in Theorem~\ref{theorem:BQ-algebra and Bott tower}.

It is proved in \cite{CMS} that the class of three-stage Bott manifolds are cohomologically rigid.
An immediate consequence of this result together with Theorem~\ref{theorem:BQ-algebra and Bott tower}
is Theorem~\ref{theorem:3 stage Bott manifold}, which says that the class of $6$-dimensional quasitoric manifolds
whose cohomolgies are BQ-algebras over $\mathbb Z$ is cohomologically rigid.

So far, all the cohomological results are over $\Z$ coefficients. But by careful observation of the proofs we can see that the same conclusion can be derived with the $2$-localized $\Z_{(2)}$-coefficients. This is treated in Section~\ref{sectoin:BQ-algebra over Z_2}.

\section{A sum of two line bundles over Bott manifolds}

Let $\{B_{j}=P( \underline{\C} \oplus \xi_{j-1} ) \mid 0 \leq j \leq n \}$ be a complex Bott tower of height $n$.
By the standard results on the cohomology of projectivised bundles, we can see that the
cohomology of $B_j$ is a free module over $H^\ast(B_{j-1})$ on generators $1$ and $x_j$ of dimension
$0$ and $2$ respectively. The ring structure of $H^\ast(B_j)$ is determined by a single relation
$$x_j^2=c_1(\xi_{j-1})x_j$$
where $x_j$ is the first Chern class of the line bundle $\gamma_j$ which is the pull-back bundle of the tautological line bundle of $P(\mathbb C\oplus \xi_{j-1})=B_j$ via the
projection $B_n\to B_j$.
Since $c_1(\xi_{j-1})\in H^2(B_{j-1})$, we can write
$$ f_j:= c_1 (\xi_{j-1}) = \sum_{i=1}^{j-1}\c{j}{i}x_i.$$
Since complex line bundles are distinguished by their first Chern classes, Bott manifold $B_n$ is determined by
the above list of integers ($\c{j}{i} : 1 \leq i < j \leq
n$).

It is convenient to organize the integers $\c{j}{i}$ into an $n
\times n$ upper triangular matrix,
\begin{equation} \label{eqn:lambda}
\Lambda = \left(%
\begin{array}{cccc}
  0 & \c{2}{1} & \cdots & \c{n}{1} \\
    & 0 & \cdots & \vdots \\
    &  & \ddots & \c{n}{n-1} \\
    &  &  & 0 \\
\end{array}
\right).
\end{equation}
We call it the  \emph{associated matrix} of the Bott tower.

One of the basic questions in vector bundle theory is to determine when two bundles with
equal characteristic classes are isomorphic.
In particular, we would like to know whether the following question is true.
Let $\xi$ and $\eta$ be sums of $k$ complex line bundles over a generalized Bott manifold $B$.
{\em Are two bundles $\xi$ and $\eta$ isomorphic if their total Chern classes are equal?}
The answer is true when $B$ is a generalized Bott tower and $\eta$ is the trivial bundle, see  \cite{CMS}.
In this section we provide two more affirmative answers to the question.
They are  Proposition~\ref{proposition:trivial bundle}
and Proposition~\ref{proposition:isomorphism of sum of two line bundle}. We first need the following lemma. We sometimes confuse Bott tower with its last stage Bott manifold when they are clear from the context.

\begin{lemma}\label{lemma:2-1}
    Let $B_n$ and $B_n'$ be two $n$ stage Bott towers. If the  associated  matrices to them are
    $$
    \left(
      \begin{array}{ccccc}
        0    & \ast   & \ast & b_1 & a_1 \\
             & \ddots & \ast & \vdots & \vdots \\
             &   & 0 & b_{n-2} & a_{n-2} \\
            &   &   & 0 & 0 \\
         &  &  & & 0 \\
      \end{array}
    \right) {\textrm and }     \left(
      \begin{array}{ccccc}
        0    & \ast   & \ast & a_1 & b_1 \\
             & \ddots & \ast & \vdots & \vdots \\
             &   & 0 & a_{n-2} & b_{n-2} \\
            &   &   & 0 & 0 \\
         &  &  & & 0 \\
      \end{array}
    \right)
    $$respectively, then $B_n$ and $B_n'$ are diffeomorphic.
\end{lemma}

Note that this lemma can be seen by the fact that $B_n$ and $B_n'$ are diffeomorphic if two associated matrices are conjugate by a permuatition matrix, see \cite{ch-ma-su08} or \cite{MP}. However, here we give a direct proof of the lemma.

\begin{proof}
Let $B$ be the $(n-2)$-stage Bott manifold with the associated  $(n-2)\times (n-2)$ matrix
$$
    \left(
      \begin{array}{ccc}
         0 & \ast & \ast  \\
         & \ddots & \ast \\
         &  & 0 \\
      \end{array}
    \right).
$$ Then $B= B_{n-2} \rightarrow B_{n-3} \rightarrow \cdots \rightarrow B_1 \rightarrow B_0$, where $B_{j} = P(\C \oplus \xi_{j-1}) \rightarrow B_{j-1}$.

Let $\gamma_j$ be the pull-back of the tautological line bundle of $P(\C \oplus \xi_{j-1})= B_j$ via the projection $B=B_{n-2}\to B_j$, and let $c_1(\gamma_j)=x_j$ for $j \le n-2$. Let $\alpha=\sum_{i=1}^{n-2}a_ix_i$ and $\beta=\sum_{i=1}^{n-2}b_i x_i \in H^2(B)$. Define two complex line bundles over $B$
$$\gamma^\alpha=\bigotimes_{i=1}^{n-2}\gamma_i^{a_i} \rightarrow B, \text{ and } \gamma^\beta=\bigotimes_{i=1}^{n-2}\gamma_i^{b_i} \rightarrow B.$$
Let $\pi_\alpha : P(\C \oplus \gamma^\alpha) \rightarrow B$ be the projection of the $\mathbb CP^1$-bundle over $B$
and denote  this fiber bundle by $\eta_{\alpha}$. Similarly   $\pi_\beta : P(\C \oplus \gamma^\beta) \rightarrow B$  and $\eta_{\beta}$ is defined. Then
\[
\xymatrix{
    B_n' \cong \pi_\alpha^\ast(\eta_{\beta}) \ar[d]& & \ar[d] \pi_\beta^\ast(\eta_{\alpha}) \cong B_n \\
    B_{n-1}' = P(\C \oplus \gamma^\alpha) \ar[dr]_{
     \pi_\alpha} & & \ar[dl]^{\pi_\beta
    }
    P(\C \oplus \gamma^\beta) =B_{n-1} \\
    & B = B_{n-2} &
},\] where $\pi_\beta^\ast(\eta_{\alpha})=\{ (x,y) \in P(\C \oplus \gamma^\beta) \times P(\C \oplus \gamma^\alpha) | \pi_\beta(x) = \pi_\alpha(y)\}$ and $\pi_\alpha^\ast(\eta_{\beta})=\{ (a,b) \in P(\C \oplus \gamma^\alpha) \times P(\C \oplus \gamma^\beta) | \pi_\alpha(a) = \pi_\beta(b)\}.$
Therefore $\pi_\beta^\ast(\eta_{\alpha}) \cong \pi_\alpha^\ast(\eta_{\beta})$.
\end{proof}

\begin{corollary}\label{corollary:last twist}
If a Bott manifold has a one-twist Bott tower structure, then it has another Bott tower structure
whose last stage is nontrivial and all other stages are trivial.
\end{corollary}
\begin{proof}
By successive applications of Lemma \ref{lemma:2-1}, we can push the trivial fibration down to lower levels.
\end{proof}

\begin{corollary}\label{corollary:columns with zeros}
If $B_n$ is a Bott tower with the associated matrix
    $$\Lambda=
\left(%
\begin{array}{cccc}
  0 & \c{2}{1} & \cdots  & \c{n}{1}   \\
    & 0 &  & \vdots \\
     &  & \ddots & \c{n}{n-1}\\
  &   &   & 0
\end{array}
\right)
$$
such that $\c{k+1}{k} = \c{k+2}{k} = \cdots = \c{n}{k} = 0$. Then $B_n$ is diffeomorphic to
a Bott tower with the associated matrix
$$\Lambda'=
    \left(%
\begin{array}{ccccccccc}
0 &\c{2}{1}& \cdots & \c{k-1}{1}& \c{k+1}{1}& \c{k+2}{1}& \cdots& \c{n}{1}& \c{k}{1} \\ %
  & 0 & \cdots & \c{k-1}{2}&\c{k+1}{2}&\c{k+2}{2}&\cdots&\c{n}{2}&\c{k}{2} \\ %
  &   & \ddots & \vdots& \vdots& \vdots & \vdots&\vdots &\vdots\\
  &   & & 0 &\c{k+1}{k-1}&\c{k+2}{k-1}&\cdots&\c{n}{k-1}&\c{k}{k-1} \\%
  &   & &  & 0 & \c{k+2}{k+1}&\cdots&\c{n}{k+1}& 0\\
  &   & &  & & 0 & \vdots& \vdots& \vdots\\
  &   & &  &  & &\ddots & \c{n}{n-1}& 0\\
  &   & &  & & & &0&0\\%
  &   & &  & & & &&0
\end{array}%
\right).
$$
\end{corollary}

Note that $\Lambda'$ is obtained from $\Lambda$ by interchanging the $k$-th and the $n$-th rows and the $k$-th and the $n$-th columns. So $\Lambda'$ is conjugate to $\Lambda$, and again, by \cite{ch-ma-su08} and \cite{MP}, we can see that their corresponding Bott manifolds are diffeomorphic. However, we give an elementary and direct indication of proof here for reader's convenience.

\begin{proof}[Sketch of Proof]
This is an easy consequence of Lemma~\ref{lemma:2-1}. The only thing to consider is
that when exchanging the columns we need to take care of the effect of the indices of $x_j$'s.
Here we only give an idea of the proof with an example.
The proof of the general case is quite similar.
Here we consider $B_4$ with the following associated matrix
$$A=
\left(
\begin{array}{c c c c}
0 & a & b & c \\
 & 0 & 0 & 0 \\
 &  & 0 & d \\
 &  &  & 0
\end{array}
\right).
$$
Then $H^\ast(B_4)\cong \mathbb Z[x_1, x_2, x_3, x_4]/I$ where $I$ is the ideal generated by
$$ x_1^2,\  x_2(x_2-ax_1),\  x_3(x_3 - bx_1), \ x_4(x_4 - dx_3 - cx_1).$$
We apply Lemma~\ref{lemma:2-1} to $B_3$ whose associated matrix is
$$B=
\left(
\begin{array}{c c c}
0 & a & b \\
 & 0 & 0 \\
 &  & 0 \\
\end{array}
\right),
$$
which results exchanging the second and third columns of $B$. The effect of the above
procedure also exchanges the second and third stages of the Bott tower of $B_4$,
and as a result the variable $x_2$ and $x_3$ will be exchanged. Namely, the variables
$x_1, x_2, x_3, x_4$ will be changed to $x_1', x_3', x_2', x_4'$.
Therefore, with the changed variables $x_1', x_2', x_3', x_4'$,
$B_4$ is diffeomorphic to $B_4'$ with the assocoated matrix
$$A'=
\left(
\begin{array}{c c c c}
0 & b & a & c \\
 & 0 & 0 & d \\
 &   & 0 & 0 \\
 &  &  & 0
\end{array}
\right).
$$
We now apply Lemma~\ref{lemma:2-1} to $A'$.
This means that we are exchanging the third and fourth stages of the Bott tower of $B_4'$
to get $B_4 ^{\prime\prime}$ with the associated matrix
$$A^{\prime\prime}=
\left(
\begin{array}{c c c c}
0 & b & c & a \\
 & 0 & d & 0 \\
 &  & 0 & 0 \\
 &  &  & 0
\end{array}
\right).
$$
with the variables $x_1', x_2', x_3', x_4'$ changed to $x_1^{\prime\prime},x_2^{\prime\prime},x_4^{\prime\prime},x_3^{\prime\prime}$.
This is the desired result.
\end{proof}

\begin{proposition}\label{proposition:trivial bundle}
A sum of two line bundles over a Bott manifold is trivial if and only if the total Chern class is trivial.
\end{proposition}
\begin{proof}
    Let $B_n$ be a Bott manifold with the associated matrix
    $$
\Lambda=\left(%
\begin{array}{cccc}
  0 & \c{2}{1} & \cdots & \c{n}{1} \\
    & 0 &        & \vdots  \\
    &   & \ddots & \c{n}{n-1}\\
    &   &   & 0
\end{array}
\right).
$$
As before, let $x_j$ be the first Chern class of the line bundle $\gamma_j$ which is the pull-back bundle
of the tautological line bundle of $P(\mathbb C\oplus \xi_{j-1})=B_j$ via the
projection $B_n\to B_j$.
Let
$f_j = \sum_{i=1}^{j-1}\c{j}{i}x_i$.
For an element $\alpha \in H^2(B_n)$, let $\gamma^{\alpha}$ be the complex line bundle over $B_n$ with $c_1(\gamma^\alpha) = \alpha$. Let $\xi =
\gamma^{\alpha} \oplus \gamma^{\beta}$ be the sum of two line bundles such that $c(\xi)=1$, and $\alpha = \sum_{j=1}^n a_j x_j$ and $\beta = \sum_{j=1}^{n} b_j x_j$. Then
\begin{align*}
1 = c(\xi) &= c(\gamma^\alpha) c(\gamma^\beta) \\
        &= (1+ \alpha)(1+\beta) \\
        &= 1 + (\alpha + \beta) + \alpha\beta.
\end{align*}
Therefore $\alpha + \beta = 0$ and $\alpha \beta = 0$, which implies $\alpha^2=0$ in $H^\ast(B_n)$.
On the other hand,
\begin{align}
    \alpha^2 = 0 &\Leftrightarrow \sum_{j=1}^{n}(a_jx_j)^2 +
    \sum_{1\leq i<j \leq n} 2a_ja_i x_j x_i = \sum_{j=1}^n a_j^2 (
    x_j^2 - f_jx_j) \nonumber\\
    & \Leftrightarrow  a_j^2 \c{j}{i} = - 2a_ja_i \text{ for all }
    i<j \label{equation:trivial chern class}
\end{align}
Thus $\xi = \gamma^{\alpha} \oplus \gamma^{-\alpha}$ with $\alpha = \sum_{j=1}^n a_j x_j \in H^2(B_n)$ and $a_j^2 \c{j}{i} = - 2a_j a_i$ for all $ 1 \leq i < j \leq n$.

Now, we prove the proposition by induction on $n$. If $n=2$, then the dimension of $\xi$ is equal to the dimension of $B_n$, so we are in the stable range. Hence the total Chern class classifies the complex vector bundle, so the proposition is true for $n=2$. Assume the lemma is true for $B_{n-1}$. We now prove the lemma for $B_n$. These are three cases to consider

\textbf{Case 1} \hspace{.5cm} $a_n = 0$.

In this case, $\xi =\pi_n^\ast(\eta)$, where $\eta =\gamma^\ma \oplus \gamma^{-\ma}$ over $B_{n-1}$. By the assumption, $c(\eta)=1$, and by the induction hypothesis, $\eta$ is trivial. So is $\xi$.

\textbf{Case 2} \hspace{.5cm} $a_n \neq 0$ and $a_k=0$ for some $k<n$.

We may assume that $a_i \neq 0$ for all $i>k$. By \eqref{equation:trivial chern class}, $a_j^2 \c{j}{i} = - 2a_j a_i$ for all $i<j$. Hence, $a_{k+\ell}^2 \c{k+\ell}{k} = -2a_{k+\ell}a_k$ for all $0< \ell \leq n-k$. Since $a_{k+\ell} \neq 0$ and $a_k=0$, $\c{k+\ell}{k}=0$ for all $\ell$. Thus, $\c{k+1}{k} = \c{k+2}{k} = \cdots = \c{n}{k} = 0$. Hence $B_n$ is diffeomorphic to a
Bott manifold $B_n'$ with $\Lambda'$ in Corollary~\ref{corollary:columns with zeros} as the associated matrix.

Let $\{x_1, \ldots, x_n\}$ and $\{y_1, \ldots, y_n \}$ be ordered generator sets of $H^\ast(B_n)$ and $H^\ast(B_n')$ respectively as in \eqref{eqn:cohomology}. Let $\rho : B_n' \rightarrow B_n$ be the diffeomorphism as indicated
in the proof of Corollary~\ref{corollary:columns with zeros}. Then we can see that
\begin{align*}
    \rho^\ast(x_1) &=y_1 \\
    \vdots &\\
    \rho^\ast(x_{k-1}) &= y_{k-1} \\
    \rho^\ast(x_k) &= y_n \\
    \rho^\ast(x_{k+1}) &=y_k \\
    \vdots &\\
    \rho^\ast(x_n) &=y_{n-1}.
\end{align*}
Therefore, $\rho^\ast(\alpha) = a_1 y_1 +
\cdots + a_{k-1}y_{k-1} + a_{k+1} y_k + \cdots + a_{n} y_{n-1} +
a_k y_n$. Since $c(\gamma^\alpha \oplus \gamma^{-\alpha})=0$ in $H^\ast(B_n)$, $c(\rho^\ast(\gamma^\alpha \oplus \gamma^{-\alpha}))=0$ in $H^\ast(B_n')$. Since $a_k=0$ from the assumption, we are in \textbf{Case 1} for $B_n'$. Therefore $\rho^\ast(\gamma^\alpha \oplus \gamma^{-\alpha})$ is trivial on $B_n'$, and so is $\gamma^\alpha \oplus \gamma^{-\alpha}$ on $B_n$.

\textbf{Case 3} \hspace{.5cm} $a_j \neq 0$ for all $j$.

By \eqref{equation:trivial chern class}, $a_j^2 \c{j}{i} = -2a_ja_i$ for all $i<j$, hence, $\c{j}{i} \neq 0$ for all $i,j$.
Note that $B_2$ is a Hirzebruch surface.
Since the diffeomorphism type of a Hirzebruch surface $B_2$ is determined by the parity of $\c{2}{1}$, if $\c{2}{1}$ is even then $B_2 $ is diffeomorphic to $\CP^1 \times \CP^1$. Hence $B_n$ is diffeomorphic to $B_n'$ with $\c{2}{1}'=0$. Thus we may assume that $\c{2}{1} =0$ for simplicity. But then, by \eqref{equation:trivial chern class} either $a_1$ or $a_2$ is zero, which contradicts
to the assumption of Case 3. Therefore we may assume that $\c{2}{1}$ is odd; in fact we may assume that $\c{2}{1}=1$ because the diffeomorphism type of $B_2$ is determined by the parity of $\c{2}{1}$. Since $\c{j}{i} \neq 0$ and $a_j \neq 0$ for all $j$ and $i$, by \eqref{equation:trivial chern class}, $\c{j}{i}a_j = -2a_i$ for all $i<j$. Hence $a_2 = -2 a_1$. Moreover since
$\c{3}{1}a_3 = -2a_1$ we have $\c{3}{2}a_3 = -2a_2 = 4a_1$, hence, $\c{3}{2}=-2\c{3}{1}$. We claim that $B_3$ with
$$
\left(
\begin{array}{ccc}
0 & 1 &c \\
 & 0 & -2c \\
 &  & 0 \\
\end{array}%
\right)
$$
is diffeomorphic to $B_3'$ with
$$
\left(
\begin{array}{ccc}
0 & 1 & c \\
 & 0 & 0 \\
 &  & 0 \\
\end{array}%
\right).
$$ Thus we may assume that $B_n$ has $\c{3}{2}=0$. Then by \eqref{equation:trivial chern class}, $a_2$ or $a_3$ must be zero. Therefore we are in \textbf{Case 2} and the proposition is proved.

It remains to prove the claim.
\begin{align*}
B_3 &= P(\C \oplus (\gamma_1^{c} \otimes \gamma_2^{-2c})) \\
&\cong P((\C \oplus (\gamma_1^{c} \otimes \gamma_2^{-2c}))\otimes \gamma_2^{c}) \\
&\cong P(\gamma_2^{c} \oplus (\gamma_1^{c} \otimes \gamma_2^{-c})). \\
\end{align*}
The total Chern class of $\gamma_2^{c} \oplus (\gamma_1^{c}\otimes\gamma_2^{-c})$ is
$$
c(\gamma_2^{c} \oplus (\gamma_1^{c} \otimes \gamma_2^{-c})) = (1+cx_2)(1+cx_1 - cx_2) = 1+ cx_1
$$
since $x_2^2 = x_1x_2$ in $ H^4(B_2)$.

On the other hand, $c(\C \oplus \gamma_1^c)=1+cx_1$. Therefore, $\gamma_2^{c} \oplus (\gamma_1^{c} \otimes \gamma_2^{-c}) \cong \C \oplus \gamma_1^{c}$ as bundles over $B_2$. Thus, $B_3 \cong P(\C \oplus \gamma_1^c)= B_3'$ which
has the associated matrix
$$
\left(
\begin{array}{ccc}
0 & 1 &c \\
 & 0 & 0 \\
 &  & 0 \\
\end{array}%
\right).
$$
\end{proof}

Now let $B_{n-1}\cong (\mathbb CP^1)^{n-1}$, and for $\alpha \in H^2(B_{n-1})$
let $\gamma^\alpha$ be the complex line bundle over $B_{n-1}$ with $c_1(\gamma^\alpha)=\alpha$
as before.

\begin{proposition}\label{proposition:isomorphism of sum of two line bundle}
Let $\xi_1=\gamma^{\alpha_1} \oplus \gamma^{\alpha_2}$ and
$\xi_2=\gamma^{\beta_1} \oplus \gamma^{\beta_2}$ be sums of two line bundles over $B_{n-1}\cong (\mathbb CP^1)^{n-1}$
such that $c_1(\xi_1)=c_1(\xi_2)$ and $c_2(\xi_1)=c_2(\xi_2)=0$.
Then $\xi_1$ and $\xi_2$ are isomorphic.
\end{proposition}

\begin{proof}
Let $H^\ast(B_{n-1})\cong \mathbb Z[x_1,\ldots, x_{n-1}]/<x_j^2\mid j=1,\ldots, n-1>$, and
let $\alpha_k, \beta_k$ be elements of $H^2(B_{n-1})$ for $k=1,2$.
From the assumption we have $\alpha_1+\alpha_2=\beta_1+\beta_2$ and $\alpha_1\alpha_2=\beta_1\beta_2=0$.

In general, for two elements $u=\sum_{i=1}^{n-1}u_ix_i$   and $v=\sum_{i=1}^{n-1}v_ix_i$  of $H^2(B_{n-1})$,
the identity $uv=0$ holds if and only if $u_jv_i+u_iv_j=0$ for any $j\neq i$. From this, we can see easily that
if $uv=0$, one of the following three possibilities follows.
\begin{enumerate}
\item If at least three coefficients in $u$ are non-zero, then $v=0$.
\item If exactly two coefficients in $u$, say $u_i$ and $u_j$, are non-zero, then so is
$v$ with $v_jv_i\neq 0$ and $u_jv_i+u_iv_j=0$.
\item If only one coefficient in $u$, say $u_j$, is non-zero, then so is $v$ with $v_j\neq 0$.
\end{enumerate}
Suppose $\alpha_1$ has at least three non-zero coefficients. Then  (1) implies that $\alpha_2=0$ and
$\alpha_1=\beta_1+\beta_2$. If $\beta_1\neq 0$ and has at most two non-zero coefficients then so is $\beta_2$
with  non-zero coefficients at the same places  as $\beta_1$ by (2) and (3), which is a contradicts to the assumption that
$\alpha_1$ has at least three nonzero coefficients because $\alpha_1+\alpha_2=\beta_1+\beta_2$.
So $\beta_1$ is either $0$ or has at most three non-zero coefficients. Therefore by (1) either $\beta_1=0$ or
$\beta_2=0$, and two bundles $\xi_1$ and $\xi_2$ are isomorphic.

Suppose that  $\alpha_1$ has  exactly  two non-zero coefficients. Then (1) and (2) imply that so is $\alpha_2$, and $\beta_1$ and $\beta_2$ are either zero or have exactly two non-zero coefficients at the same places as $\alpha_1$. This means that the bundles $\xi_1$ and $\xi_2$ are pullbacks of bundles over $(\mathbb CP^1)^2$.
Hence those bundles are in stable range and hence they are classified by their Chern classes. Thus $\xi_1$ and
$\xi_2$ are isomorphic.

The case when $\alpha_1$ has only one non-zero coefficient can be proved similarly.

\end{proof}

\section{Twist number and Cohomological complexity} \label{section: twist number}

The {\em twist number}  of a Bott tower $\{B_j \mid j=0,\ldots,n\}$ is the number of nontrivial
fibrations $B_j\to B_{j-1}$ in the sequence. However  there may be several Bott tower
structures for a Bott manifold, so the twist number may not be well-defined for Bott manifolds.
In this section we show that
the twist number of a Bott manifold is  well-defined, namely we show that the twist numbers of any Bott tower
structure of a Bott manifold is constant.

For an $n$-stage Bott manifold $M$ its cohomology ring is isomorphic to
$$
    H^\ast(M) \cong \Z[x_1, \ldots, x_n]/I,
$$ where $I=<x_j(x_j - f_j) : j=1, \ldots, n>$ and $f_j = \sum_{i=1}^{j-1} \c{j}{i}x_i$
with $\deg x_j=2$.
Here the numbers $\c{j}{i}$ can be determined by a Bott tower structure of $M$. Indeed, $\c{j}{i}$'s are
the entries of the matrix (\ref{eqn:lambda}). Hence if the fibration
of the $i$-th stage of a Bott tower structure on $M$ is trivial,
then we may assume that $f_j=0$. Therefore the number of nonzero $f_j$'s  may depend not only on the choices of
generators of the cohomology ring $H^\ast(M)$ but also the Bott tower structures of $M$.
The {\em cohomological complexity of $M$} is the minimal number of nonzero $f_j$'s
among all possible such choices.

In the following theorem we show that the twist number of any Bott tower structure of a Bott manifold
$M$ is equal to the cohomological complexity of $M$. This, in particular, shows that the twist number
of a Bott manifold is well-defined.

\begin{theorem} \label{theorem:twist number = complexity}
    Let $M$ be a Bott manifold. Then the twist number  of any Bott tower structure of  $M$ is equal
    to the cohomological complexity of $M$.
\end{theorem}
\begin{proof}
Let
$$
B_n \rightarrow B_{n-1} \rightarrow \cdots \rightarrow B_1 \rightarrow B_0 = \{\text{a point}\}
$$ be a Bott tower structure of $M$  whose twist number is equal to  $t$.
By Corollary \ref{corollary:last twist}, we may assume that
$$B_{n-t} \rightarrow B_{n-t-1} \rightarrow \cdots \rightarrow B_1$$
is a trivial Bott tower. Therefore $B_{\ell} = (\CP^1)^{\ell}$ for $\ell=1, \ldots, n-t$.

Let $s$ be the cohomological complexity of $M$. Then it is clear that $t\geq s$ in general. Suppose $t>s$. Since the twist number of $M$ is $t$, we have
$$
        H^\ast(B_n) = \Z[x_1, \ldots, x_n] / <x_j(x_j - f_j) \mid j=1,\ldots,n>,
$$
where
$$f_j = \left\{%
\begin{array}{ll}
    0 & \text{for $1 \leq j \leq n-t$} \\
    \sum_{i=1}^{j-1} \c{j}{i}x_i & \text{for $ n-t < j \leq n$}. \\
\end{array}%
\right.
$$

Since the cohomological complexity of $B_n$ is $s$, there is an isomorphism
$$
    \psi : H^\ast(B_n) \rightarrow \Z[y_1, \ldots, y_n] / <y_j(y_j - g_j) \mid j=1,\ldots,n>,
$$ where
$$g_j = \left\{%
\begin{array}{ll}
    0 & \text{for $1 \leq j \leq n-s$}, \\
    \sum_{i=1}^{j-1} \d{j}{i}y_i & \text{for $n-s< j \leq n$.}
\end{array}%
\right.$$

\textbf{We claim that there exists $m$ ($n-t < m \leq n$) such that $f_m
\equiv 0 \mod 2$ and $f_m^2 = 0 \in H^\ast(B_{m-1})$.}\newline
If the claim is true, then we can write as $f_m + 2w = 0$ for some $w \in H^2(B_{m-1})$. Therefore,
\begin{align*}
c(\gamma^w \oplus \gamma^{f_m + w}) &= (1+w)(1+f_m+w) = 1+ (f_m + 2w) - \frac{f_m^2}{4} \\
& = 1
\end{align*} Thus, by Proposition \ref{proposition:trivial bundle}, $\gamma^w \oplus \gamma^{f_m+w}$ is a trivial bundle over $B_{m-1}$. Hence $P(\C \oplus \gamma^{f_m})=P(\gamma^w \oplus \gamma^{f_m+w})=B_{m-1}\times \CP^1$. So we can reduce the twist number of $B_n$ to $t-1$, which is a contradiction.

\textbf{We now prove the claim.}
Since $\psi$ is an isomorphism, we can write
$$
    y_i = \sum_{j=1}^n b_{ij} \psi(x_j).
$$
Let  $B = (b_{ij})$ be the coefficient matrix. Note that $\det(B) = \pm1$.

Since $\psi^{-1}(y_k^2) = 0$ in $H^\ast(B_n)$ for $1 \leq k \leq n-s$, we have
\begin{align}
    \psi^{-1}(y_k^2) &= (\psi^{-1}(y_k))^2 = (\sum_{j=1}^n b_{kj}x_j)^2 \nonumber \\
    & = \sum_{j=1}^n (b_{kj})^2x_j^2 + \sum_{1\leq i<j\leq n} 2b_{kj}b_{ki}x_ix_j \label{eqn:6}\\
    &= \sum_{j=1}^{n}(b_{kj})^2(x_j^2 - f_jx_j) \text{, which represents zero in $H^\ast(B_n)$}  \label{eqn:7} \\
    &= \sum_{j=1}^{n}(b_{kj})^2(x_j^2 - \sum_{i=1}^{j-1}\c{j}{i}x_ix_j)\nonumber
\end{align}
By comparing the coefficients of \eqref{eqn:6} and \eqref{eqn:7}, we have
\begin{equation}
    \sum_{i=1}^{j-1} 2b_{kj}b_{ki}x_i =
    -(b_{kj})^2f_j \label{eqn:8}
\end{equation} for $ 1\leq k \leq n-s$ and $1\leq j \leq n$. This implies
\begin{equation} \label{eqn:9}
    2b_{kj}b_{ki} = -(b_{kj})^2 \c{j}{i}
\end{equation}
where $ 1\leq k \leq n-s$ and $ 1\leq i <
j$.

Suppose that all $b_{kj}$ are even for $ n-t+1 \leq j \leq n$
and $ 1\leq k \leq n-s$. Since $t + n - s > n$, $\det B$ must be
even because, in general, if
$A = \left(
\begin{array}{cc}
C & D \\
E & F \\
\end{array}
\right)$
is an $n \times n$ matrix and if $D$ is a $k \times \ell$ matrix all of whose entries are even with $k + \ell >n$, then $\det B$ is even.
This is a contradiction. Thus there is an odd number $b_{\ell m}$ for some $ 1 \leq \ell \leq n-s$ and $ n-t+1 \leq m \leq n$.

Suppose $f_m$ is not congruent to 0  modulo 2, i.e., there exists an odd number $\c{m}{h}$ for some $1\leq h \leq m-1$. Then from \eqref{eqn:9}, $2b_{km}b_{kh} = -(b_{km})^2 \c{m}{h}$ for $1 \leq k \leq n-s$. It implies that $b_{km} \equiv 0 \text{ (mod $2$)}$ for all $1 \
\leq k \leq n-s$, which contradicts to that $b_{\ell m}$ is odd. Thus, $f_m \equiv 0 \text{ (mod $2$)}$.

On the other hand, from
\eqref{eqn:8}, $\sum_{i=1}^{m-1} 2b_{km}b_{ki}x_i = -(b_{km})^2f_m$ with $k=\ell$, $\frac{f_m}{2} = - \sum_{j=1}^{m-1} \frac{b_{\ell j}}{b_{\ell m}}x_j$. Thus we have
\begin{align*}
    \left( \frac{f_m}{2} \right)^2 &= \left(- \sum_{j=1}^{m-1} \frac{b_{\ell j}}{b_{\ell m}}x_j\right)^2 \\
    &= \sum_{j=1}^{m-1} \left(\left(\frac{b_{\ell j}}{b_{\ell m}}\right)^2 x_j^2 +2
 \sum_{h=1}^{j-1} \frac{b_{\ell j}b_{\ell h}}{(b_{\ell m})^2}x_j x_h \right) \\
    &= \frac{\sum_{j=1}^{m-1} (b_{\ell j})^2 (x_j^2 - f_j x_j)}{(b_{\ell m})^2} \hspace{2cm} \text{by \eqref{eqn:8}}\\
    &=0 \in H^\ast(B_{n-1})
\end{align*}

This proves the claim.
\end{proof}

From the proof of Theorem~\ref{theorem:twist number = complexity} the following corollary follows immediately.
\begin{corollary}\label{coro:well definedness of twist number}
The twist number of a Bott manifold $M$ is well-defined, i.e., any two Bott tower structures of $M$ have
the same twist number.
\end{corollary}

\section{BQ-algebras and Bott manifolds}
Recall that a $2n$-dimensional manifold $M$ is a quasitoric manifold over a simple (combinatorial) polytope $P$
if there is a locally standard $n$-torus $T^n$ action on $M$ and a surjective map $\pi:M\to P$ whose
fibers are the $T^n$-orbits.  For a $2n$-dimensional quasitoric manifold $M$ over a simple polyotpe $P$
there  corresponds a {\em characteristic map} $\chi:\mathcal F \to \mathbb Z^n$ well-defined up to sign
where $\mathcal F$ is the
set of all facets of $P$. A characteristic map should satisfy the following two conditions:
\begin{itemize}
\item $\chi(F)$  is a  primitive vector for any $F\in\mathcal F$, and
\item  if  $n$ facets $F_1\ldots, F_n$ are intersecting at vertex $v$ of $P$, then
$\{\chi(F_1),\ldots, \chi(F_n)\}$ forms a linearly independent subset in $\mathbb Z^n$.
\end{itemize}
Conversely, for simple polytope $P$ and a map  $\chi:\mathcal F \to \mathbb Z^n$ satisfying the above two
conditions, there exists a unique quasitoric manifold up  equivalence whose characteristic map is $\chi$.

Two quasitoric manifolds $\pi_M: M\to P$ and $\pi_N:N\to P$ over $P$ are {\em equivalent} if there is a weak $T^n$- equivariant
homeomorphism $\phi :M\to N$ (i.e., there exists an automorphism $\rho$ on $T^n$ such that $\phi(tx)=\rho(t)\phi(x)$)
such that $\pi_N\circ\phi=\pi_M$.

Let $P$ be an $n$-dimensional simple polytope with $m$ facets, and let $M$ be a quasitoric manifold over $P$.
Then  we can find a characteristic map   $\chi$ for $M$  such that
$\chi(F_1)=(1,0,\ldots,0),\ldots, \chi(F_n)=(0,\ldots,0,1)$ where $F_1,\ldots, F_n$ are the facets meeting at one
particular vertex $p\in P$. Then we can define an $(m-n)\times n$ matrix $A$ whose row vectors are $\chi(F_{n+1}), \ldots, \chi(F_{m})$. This matrix $A$
is called a {\em characteristic matrix} of $M$. For the details about quasitoric manifolds we refer the reader to \cite{DJ}. We note that a Bott manifold $B_n$ associated with the matrix $\Lambda$ in \eqref{eqn:lambda} admits the canonical nice $T^n$-action with which $B_n$ becomes a quasitoric manifold. The characteristic matrix of $B_n$ is then equal to $-\Lambda - I_n$, where $I_n$ is the identity matrix of size $n$, see \cite{MP} for details.

In this section we will consider quasitoric manifolds whose cohomology rings resemble those of Bott manifolds.
For this we need the following definition.
\begin{definition} \label{def:BQ-algebra}
A graded algebra $S$ over $\mathbb Z$  generated by $x_1,\ldots,x_n$
of degree $2$ is called a Bott quadratic algebra (BQ-algebra) over $\mathbb Z$
of rank $n$ if
\begin{enumerate}
\item $x_k^2=\sum_{i<k}\c{k}{i}x_ix_k$ where $\c{k}{i}\in\mathbb Z$ for $1\le k\le n$, (in particular $x_1^2=0$,) and
\item $\prod_{i=1}^n x_i\ne 0$.
\end{enumerate}
BQ-algebra over $\mathbb Z_2$ is defined similarly.
\end{definition}
Originally, BQ-algebra over $\mathbb Z_2$ is  defined in \cite{MP}, and  we extend their definition here for our purpose.
The cohomology ring of a Bott manifold is a BQ-algebra over $\mathbb Z$.
So one might ask whether the converse is true, i.e., if the cohomology ring of
a quasitoric manifold is a BQ-algebra over $\mathbb Z$, then is the quasitoric manifold homeomorphic
to a Bott tower? The affirmative and stronger answer to the question is given in the following theorem.

\begin{theorem}\label{theorem:BQ-algebra and Bott tower}
Let $M$ be a $2n$-dimensional quasitoric manifold over a simple polytope $P$, and
let $A$ be a characteristic matrix of $M$.
Then the following are equivalent.
\begin{enumerate}
\item $M$ is equivalent to an $n$-stage Bott manifold.
\item $H^\ast(M)$ is a BQ-algebra of rank $n$ over $\mathbb Z$.
\item $P$ is combinatorially equivalent to the cube $I^n$ and  $A$ is  conjugate
to an upper triangular matrix by a permutation matrix.
\end{enumerate}
\end{theorem}

\begin{proof}
(1) $\Rightarrow$ (2)  Clear.

(3) $\Leftrightarrow$ (1) follows from Proposition~3.2 in \cite{MP}.

(2) $\Rightarrow$ (3) If $H^\ast(M)$ is a BQ-algebra of rank $n$ over $\mathbb Z$,
then $H^\ast(M:\mathbb Z_2)$ is a BQ-algebra of rank $n$ over $\mathbb Z_2$.
By \cite[Theorem 5.5]{MP} (or \cite[Theorem 1.6]{CPS})
$P$ is combinatorially equivalent to the cube $I^n$.
Therefore $A$ is an $n \times n$ matrix. We may assume that
$$-A=
\left(
  \begin{array}{cccc}
    1 & a_{12} & \cdots & a_{1n} \\
    a_{21} & 1 & \cdots & a_{2n} \\
    \vdots & \vdots & \ddots & \vdots \\
    a_{n1} & a_{n2} & \cdots &1 \\
  \end{array}
\right)
.$$ We note that the conditions of a characteristic map implies that all principal minors are $\pm 1$, and by general facts on the cohomology of quasitoric manifolds we have an isomorphism
$$H^\ast(M)\cong \mathbb Z[y_1,\ldots, y_n]/<g_j\mid j=1,\ldots, n>$$
where $g_j=y_j\sum_{i=1}^n\a{j}{i}y_i$ and $\a{j}{j}=1$ for all $i=1,\ldots, n$.
Since $H^\ast(M)$ is a BQ-algebra over $\mathbb Z$ there is a $\Z$-algebra isomorphism
$$\phi\colon H^\ast(M)\to \mathbb Z[x_1,\ldots,x_n]/<x_j(x_j-f_j)\mid j=1,\ldots, n>,
$$
where $f_j=\sum_{i=1}^{j-1}\c{j}{i}x_i$.
Therefore
$$x_i=\sum_{j=1}^n  b_{ij} \phi (y_j)$$
with $\det B=\pm 1$
where $B$ is the $n\times n$ matrix $(b_{ij})$.
Since all principal $2\times 2$ minors of $A$ are $\pm 1$ by the conditions of characteristic map, we have
$1-\a{j}{i}\a{i}{j}=\pm 1$ for all $i\ne j$.

\textbf {We first claim that $\a{j}{i}\a{i}{j}=0$ for all $i\ne j$.}
Assume otherwise. Then $\a{s}{t}\a{t}{s}=2$ for some $s$ and $t$.
Since $\phi^{-1}(x_1^2)=(\sum_{j=1}^n b_{1j}y_j)^2=0$ in $H^\ast(M)$,
we have
\begin{equation}\label{equation:b1i}
\left( \sum_{j=1}^n b_{1j} y_j \right)^2 = \sum_{j=1}^n b_{1j}^2 \left( y_j \sum_{i=1}^n \a{j}{i} y_i \right).
\end{equation}
Compare the coefficients of $y_sy_t$-terms on both sides of the
equation~(\ref{equation:b1i}) to get
\begin{equation}\label{equation:2b1sb1t}
2b_{1s}b_{1t}=b_{1s}^2\a{s}{t} + b_{1t}^2\a{t}{s}
\end{equation}
Since $\a{t}{s}\a{s}{t}=2$ we have $(\a{t}{s}, \a{s}{t})= \pm(1, 2)$ or $\pm(2, 1)$. Therefore, the equation \eqref{equation:2b1sb1t} is equivalent to either $ (b_{1s} \pm b_{1t})^2 + b_{1t}^2 = 0$ or $ (b_{1s} \pm b_{1t})^2 + b_{1s}^2 = 0$. The only real solutions for equation~(\ref{equation:2b1sb1t})  is $b_{1s}=b_{1t}=0$.
Hence $\phi^{-1}(x_1)=\sum_{j\ne s,t}b_{1j}y_j$.

We now consider the second relation $x_2(x_2-f_2)$ of the BQ-algebra.
Here $f_2=\c{2}{1}x_1$.
Then $\phi^{-1}(f_2)=\phi^{-1}(\c{2}{1} x_1)=\c{2}{1}\phi^{-1}(x_1)$ has
no  $y_s$ and $y_t$-terms. Note that
\begin{align*}
\phi^{-1}(x_2(x_2-f_2))&=\phi^{-1}(x_2)^2 - \phi^{-1}(x_2)\phi^{-1}(f_2)\\
&=(\sum_{j=1}^n b_{2j}y_j)^2 - (\sum_{j=1}^n b_{2j}y_j)\c{2}{1}(\sum_{j\ne s,t}b_{1j}y_j)\\
&=0\in H^\ast(M).
\end{align*}
Therefore we have the following equation.
\begin{equation}\label{equation:b2jyj}
(\sum_{j=1}^n b_{2j}y_j)^2-(\sum_{j=1}^n b_{2j}y_j)\c{2}{1}(\sum_{j\ne s,t}b_{1j}y_j) = \sum_{j=1}^n \alpha_j g_j
\end{equation} for some $\alpha_j \in \Z$ with $j=1, \ldots, n$.
Since the second term of the left hand side of the equation \eqref{equation:b2jyj} has no
$y_sy_t$-term, no $y_s^2$-term and no $y_t^2$-term, by comparing the coefficients of $y_s^2$ and $y_t^2$  we can see that  $\alpha_s = b_{2s}^2$ and $\alpha_t = b_{2t}^2$. Hence by comparing the coefficients of $y_sy_t$ of
equation~(\ref{equation:b2jyj}) we get
$$
2b_{2s}b_{2t}=\a{s}{t}\alpha_s+\a{t}{s}\alpha_t=\a{s}{t}b_{2s}^2+\a{s}{t}b_{2t}^2
$$ which is of the same form as in equation \eqref{equation:2b1sb1t}. Hence, $b_{2s}=b_{2t}=0$. Note that $\phi^{-1}(f_3)$ also has no $y_s$ and $y_t$-terms.
Thus by the same argument as above, we can see that
$b_{3s}=b_{3t}=0$.
Continue the similar argument for $x_i(x_i - f_i)$ to get
$$b_{is}=b_{it}=0 \quad\textrm{for all } i=1,\ldots, n.$$
This implies that the $s$-th and $t$-th rows of the matrix $B$ are zero,
which implies $\det B=0$. This is a contradiction.
Therefore the claim that $\a{j}{i}\a{i}{j}=0$ for all $i\ne j$ is proved.

\textbf{We now claim that all principal minors of $A$ are $1$ }
by induction on the rank of the minors. By the previous claim, any principal minor of rank $2$ is 1.
Assume the claim is true for all principal minors of $\rank <k$ with
$k\ge 3$. Suppose there exists a negative principal minor $\Xi$ of
rank $k$. Since all proper minors of $\Xi$ is $1$ and $\Xi=-1$,
by Lemma~3.3 of \cite{MP} we have
$$-1=\Xi=\det\left(
\begin{array}{ccccc}
1 & h_{j_1} & 0 & \dots & 0\\
0 & 1 & h_{j_2}& \dots & 0\\
\vdots & & \ddots & &\vdots \\
0 & \dots & \dots& 1 & h_{j_{k-1}}\\
h_{j_k} & 0 & & \cdots & 1
\end{array}\right)
$$
where $h_{i_i}\ne 0$ for all $i=1,\ldots, k$.
Consider the equation~(\ref{equation:b1i}) again, but now
compare the coefficients of $y_{j_i}y_{j_{i+1}}$ where $y_{j_{k+1}}=y_{j_1}$
for convenience.
Then we have the relation
\begin{equation}\label{equation:2bijib1ji+1}
2b_{1j_i}b_{1j_{i+1}}=h_{j_i}b_{1j_{i}}^2.
\end{equation}
Suppose one of $b_{1j_{i}}$ for $i=1,\ldots, k$ is zero. Then from \eqref{equation:2bijib1ji+1} all others must be zero, too. By a similar argument applied to the second relation $\phi^{-1}(f_2)=\c{2}{1}\phi^{-1}(x_1)$,
we can see that $b_{\ell j_i}=0$ for all $\ell=1,\ldots, n$ and $i=1,\ldots, k$.
Thus $\det B=0$, which is a contradiction.
Therefore all $b_{1j_i}$ are nonzero for $i=1,\ldots,k$, and hence, so are $h_{j_{i}}$'s.
Then
$-1=\Xi=1+(-1)^k\prod_{i=1}^k h_{j_i}$. Thus $(-1)^k\prod_{i=1}^k h_{j_i}=-2$.
By multiplying each side of equation~(\ref{equation:2bijib1ji+1}) for all $i=1,\ldots, k$,
we have
$$2^k(\prod_{i=1}^k b_{1j_i})^2=(-1)^{k+1}2(\prod_{i=1}^kb_{1j_i})^2,$$ which is
a contradiction. This proves the claim.

Therefore the theorem follows from Lemma~3.3 of \cite{MP}.
\end{proof}

It is shown in \cite{CMS} that three-stage Bott manifolds are cohomologically rigid, i.e.,
if $M$ and $N$ are two three-stage Bott manifolds whose cohomology rings are isomorphic, then
they are diffeomorphic. The following corollary shows  the cohomological rigidity of the class
of $6$-dimensional quasitoric
manifolds whose cohomology rings are BQ-algebras over $\mathbb Z$.

\begin{theorem}\label{theorem:3 stage Bott manifold}
Let $M$ and $N$ be $6$-dimensional  quasitoric manifolds whose cohomology rings
are BQ-algebras over $\mathbb Z$. If $H^\ast(M)\cong H^\ast(N)$ as
graded rings, then $M$ and $N$ are diffeomorphic.
\end{theorem}
\begin{proof}
Since $H^\ast(M)$ and  $H^\ast(N)$ are BQ-algebra over $\mathbb Z$,
$M$ and $N$ are equivalent to $6$-dimensional  Bott manifolds. In particular they are hoemomorphic to
$6$-dimensional  Bott manifolds. Since all quasitoric manifolds are simply connected, by the result of Wall \cite{Wa} and Juppe \cite{Ju}, we can see that $M$ and $N$ are actually diffeomorphic to $6$-dimensional Bott manifolds. Hence the corollary follows from the above mentioned result of
\cite{CMS}.
\end{proof}

\section{Cohomological rigidity of one-twist Bott manifolds}
In this section we prove the cohomological rigidity of one-twist Bott manifolds.
Let  $\{B_{j} \mid 0 \leq j \leq n \}$ be a one-twist Bott tower.
By Corollary~\ref{corollary:last twist} we may assume that $B_{n-1}=(\mathbb CP^1)^{n-1}$.
Hence $H^2(B_{n-1})\cong
\mathbb Z[x_1,\ldots,x_{n-1}]/<x_j^2\mid x=1,\ldots, n-1>$.
Let
$M(\alpha)=B_n=P(\mathbb C\oplus \gamma^{\alpha})$ where $\gamma^{\alpha}$ is the line bundle over $B_{n-1}$
with the first Chern class
$$c_1(\gamma^{\alpha})=\alpha=\sum_{i=1}^{n-1}a_ix_i\in H^2(B_{n-1}).$$

\begin{theorem}\label{theorem:rigidity of one-twist Bott tower}
Let $\alpha$ and $\beta$ be two elements of $H^2(B_{n-1})$ where
$B_{n-1}=(\mathbb CP^1)^{n-1}$, and let $M(\alpha)$ and $M(\beta)$ be one-twist Bott manifolds as defined above.
Then
the following are equivalent.
\begin{enumerate}
\item $M(\alpha)$ and $M(\beta)$ are diffeomorphic.
\item $H^\ast(M(\alpha))\cong H^\ast(M(\beta))$ as graded rings.
\item There is an automorphism $\phi$ of $H^\ast(B_{n-1})$ such that $\phi(\alpha) \equiv \beta$ mod $2$ and
$\phi(\alpha^2)=\beta^2$.
\item Let $\alpha=\sum_{i=1}^{n-1}a_ix_i$ and $\beta=\sum_{i=1}^{n-1}b_ix_i$.Then there is a permutation $\sigma$ on $\{1,\ldots, n-1\}$ such that $a_{\sigma(i)} \equiv b_i$ mod $2$ for any
$i$ and $|a_{\sigma(i)}a_{\sigma(j)}|=|b_ib_j|$ for any $i\ne j$.
\end{enumerate}
Moreover, any isomorphism between $H^\ast(M(\alpha))$ and $H^\ast(M(\beta))$ preserves the total Pontrjagin classes of $M(\alpha)$ and $M(\beta)$.
\end{theorem}

Before we prove the theorem let us note that
\begin{equation}\label{equation:cohomology of M(a)}
H^\ast(M(\alpha))=\mathbb Z[z_1,\ldots, x_{n-1}, y_\alpha]/<x_1^2, \ldots, x_{n-1}^2, y_\alpha^2 -\alpha y_\alpha>
\end{equation}
where $y_\alpha$ is the first Chern class of the tautological bundle of $P(\mathbb C\oplus\gamma^a)$. Moreover its total Pontrjagin class is
\begin{align}
    P(M(\alpha)) &= (1 + y_\alpha)^2(1+ (y_\alpha - \alpha)^2) \nonumber \\
    & = 1 + \alpha^2. \label{eqn : Pontrjagin}
\end{align}
We first need the following lemma.

\begin{lemma}\label{lemma:rational coefficient cohomology}
The following are equivalent.
\begin{enumerate}
\item $H^\ast(M(\alpha)\colon \mathbb Q)\cong H^\ast((\mathbb CP^1)^n\colon \mathbb Q)$.
\item There is an element $u\in H^\ast(B_{n-1}\colon \mathbb Q)$ such that $(y_\alpha+u)^2=0$ in $H^\ast(M(\alpha)\colon\Q)$.
\item $\alpha=a_ix_i$ for some $i=1,\ldots, n-1$.
\end{enumerate}
Moreover there are two diffeomorphism types in this case, and $H^\ast(M(\alpha))\cong H^\ast((\mathbb CP^1)^n)$
if and only if $a_i$ is even in (3) above.
\end{lemma}

\begin{proof}
(1)$\Rightarrow$(2) Since there are $n$ linearly independent elements in the vector space
$H^2((\mathbb CP^1)^n\colon \mathbb Q)$
whose squares are zero, so are $H^2(M(\alpha)\colon \mathbb Q)$. From (\ref{equation:cohomology of M(a)})
there are $n-1$ linearly independent elements $x_1,\ldots, x_{n-1}$ in $H^2(M(\alpha)\colon\mathbb Q)$
whose squares are zero. Thus there is one more linearly independent element
$w=\sum_{i=1}^{n-1}c_{i}x_i+c_n y_\alpha\in H^2(M(\alpha)\colon\mathbb Q)$ such that $w^2=0$.
Since $w$ is linearly independent from $x_1,\ldots, x_{n-1}$, the coefficient $c_n$ of $y_\alpha$
is non-zero. Let $u=c_n^{-1}(\sum_{i=1}^{n-1}c_ix_i)\in H^2(B_{n-1}\colon \mathbb Q)$.
Then $(y_\alpha+u)^2=(c^{-1}_n)^2w^2=0$.

(2)$\Rightarrow$(3)
Let $u=\sum_{i=1}^{n-1}d_i x_i$ such that $(y_\alpha+u)^2=0$.
Then

\begin{align*}
0 & = (y_\alpha+\sum_{i=1}^{n-1}d_i x_i)^2\\
 & = y_\alpha^2+2\sum_{i<j}^{n-1}d_i d_j x_i x_j + 2\sum_{i=1}^{n-1}d_i x_iy_a\\
 & = 2\sum_{i<j}^{n-1}d_i d_j x_i x_j + (2\sum_{i=1}^{n-1}d_i x_i +\alpha)y_\alpha.
\end{align*}

This implies that $d_i d_j=0$ for all $i\ne j$, and $2\sum_{i=1}^{n-1}d_i x_i + \alpha=0$.
From the first condition at most one, say $d_i$ is non-zero.
From the second condition we have $0=2 d_ix_i + \alpha$. If we set $a_i=-2d_i$, then (3) follows.

(3)$\Rightarrow$(1)
If $\alpha=a_ix_i$, then $M(\alpha)$ is diffeomorphic to $B_2\times (\mathbb CP^1)^{n-2}$ where
$B_2=P(\mathbb C\oplus\gamma^{a_i})\to \mathbb CP^1$.
Here $\gamma$ is the tautological line bundle over
$\mathbb CP^1$. But it is well-known that there are exactly two diffeomorphism type of $B_2$
depending on the parity of $a_i$. Namely, if $a_i$ is even, then $B_2\cong (\mathbb CP^1)^2$
and if $a_i$ is odd, then $B_2$ diffeomorphic to a Hirzebruch surface $\mathcal H$.
In the former case, $H^\ast(M(\alpha)\colon \mathbb Q)$ is trivially isomorphic to $H^\ast((\mathbb CP^1)^n\colon\mathbb Q)$,
and in the latter case
\begin{align*}
H^\ast(\mathcal H\colon \mathbb Q)& \cong \mathbb Q[x_1, x_2]/<x_1^2, x_2^2-x_1x_2>\\
& \cong \mathbb Q[x_1, x_2]/<x_1^2, (x_2-\frac12x_1)^2>\\
& \cong H^\ast(\mathbb (CP^1)^2\colon \mathbb Q),
\end{align*}
which proves the lemma.
\end{proof}

We now prove Theorem~\ref{theorem:rigidity of one-twist Bott tower}
\begin{proof}[Proof of Theorem~\ref{theorem:rigidity of one-twist Bott tower}]

(2)$\Rightarrow$(3)
Let $\phi\colon H^\ast(M(\alpha))\to H^\ast(M(\beta))$ be an isomorphism.
In the case when $H^\ast(M(\alpha)\colon \mathbb Q)\cong H^\ast((\mathbb CP^1)^n\colon \mathbb Q)$,
Lemma~\ref{lemma:rational coefficient cohomology} shows that there are only two diffeomorphism types for $M(\alpha)$,
which are $(\mathbb CP^1)^n$ and $\mathcal H\times (\mathbb CP^1)^{n-2}$ where $\mathcal H$ is the
Hirzebruch surface. For these two types we can see easily that (2)$\Rightarrow$(3).

Therefore we may assume that $H^\ast(M(\alpha)\colon \mathbb Q)\cong H^\ast(M(\beta)\colon \mathbb Q)$ is not
isomorphic to $H^\ast((\mathbb CP^1)^n\colon \mathbb Q)$.
For each $x_i\in H^2(B_{n-1})\subset H^2(M(\alpha))$ we have $\phi(x_i)^2=0$ in $H^\ast(M(\beta))$.
On the other hand since $x_1, \ldots, x_{n-1}, y_\beta$ are generators of $H^\ast(M(\beta))$ we can
write $\phi(x_i)=b_{i1}x_1+\cdots + b_{i n-1}x_{n-1}+b_{in}y_\beta$. Then
by Lemma~\ref{lemma:rational coefficient cohomology} the coefficient $b_{in}$ of $y_\beta$ must vanish for $i=1,\ldots, n-1$.
This means that any isomorphism $\phi\colon H^\ast(M(\alpha))\to H^\ast(M(\beta))$ must preserve
the subring $H^\ast(B_{n-1})$. Therefore $\phi(y_\alpha)=\pm y_\beta+w$ for some $w\in H^2(B_{n-1})$.
If necessary, by composing $\phi$ with an automorphism  of $H^\ast(M(\beta))$ fixing $H^\ast(B_{n-1})$ and sending
$y_\beta$ to $-y_\beta$, we may assume that $\phi(y_\alpha)=y_\beta+w$.
It follows that
\begin{equation}\label{equation:phi-1}
\phi(y_\alpha^2)=(y_\beta+w)^2=y_\beta^2+2w y_\beta+w^2=(\beta+2w)y_\beta+w^2.
\end{equation}
On the other hand we have
\begin{equation}\label{equation:phi-2}
\phi(y_\alpha^2)=\phi(\alpha y_\alpha)=\phi(\alpha)(y_\beta+w).
\end{equation}
Comparing (\ref{equation:phi-1}) and (\ref{equation:phi-2}), we obtain
\begin{equation}\label{equation:phi-3}
\phi(\alpha)=\beta + 2w \quad\textrm{ and} \quad w^2 = \phi(\alpha)w.
\end{equation}
The first equation of (\ref{equation:phi-3}) implies $\phi(\alpha)\equiv \beta$ mod $2$.
By plugging the first equation into the second of (\ref{equation:phi-3}) we can see that
$\beta w=-w^2$. Hence $\phi(\alpha^2)=(\beta+2w)^2=\beta^2+4\beta w+4w^2=\beta^2$. Hence (2)$\Rightarrow$(3) is proved.

(3)$\Rightarrow$(2)
Suppose there is an automorphism $\phi$ on $H^\ast(B_{n-1})$ such that
$\phi(\alpha)\equiv \beta$ mod $2$ and $\phi(\alpha^2)=\beta^2$.
Let $\phi(\alpha)=\beta+2w$ for some $w\in H^2(B_{n-1})$.
If we define $\phi(y_\alpha)=y_\beta+w$, then we can see easily that $\phi$ defines
an isomorphism from $H^\ast(M(\alpha))$ to $H^\ast(M(\beta))$. This proves (3)$\Rightarrow$(2).

(1)$\Rightarrow$(2) This implication is obvious.

(2)$\Rightarrow$(1)
Suppose $H^\ast(M(\alpha))$ is isomorphic to $H^\ast(M(\beta))$. From the implication
(2)$\Rightarrow$(3), there is an automorphism $\phi$ on $H^\ast(B_{n-1})\cong \mathbb Z[x_1,\ldots,x_{n-1}/<x_j^2 \mid j=1,\ldots, n-1>$. But it is easy to see that any automorphism on
$\mathbb Z[x_1,\ldots,x_{n-1}/<x_j^2 \mid j=1,\ldots, n-1>$ is generated by a permutation on the
generators $x_1, \ldots, x_{n-1}$ and possibly changing their signs. Such automorphism on
the ring $H^\ast(B_{n-1})$ is clearly induced by a self-diffeomorphism $f$ on $B_{n-1}=(\mathbb CP^1)^{n-1}$, i.e.,
$f^\ast=\phi$. The diffeomorphism $f$ induces a fiber bundle isomorphism between $\gamma^\alpha$ and $f^\ast(\gamma^\alpha)$, hence
it induces a diffeomorphism between $M(\alpha)$ and $M(\phi(\alpha))$. Therefore for simplicity
we may assume the automorphism  $\phi$ on $H^\ast(B_{n-1})$ is the identity, such that $\alpha \equiv \beta$ mod $2$ and $\alpha^2=\beta^2$.

Since $\alpha \equiv \beta$ mod $2$, there is an element $w\in H^2(B_{n-1})$ such that $2w=\alpha-\beta$.
Now let $\xi_1=\gamma^\alpha \oplus\mathbb C$ and $\xi_2=\gamma^{w}(\gamma^\beta \oplus\mathbb C)$. Then their first Chern classes are equal because
$c_1(\xi_1)=\alpha=\beta+2w=c_1(\xi_2)$. Their  second Chern classes are $c_2(\xi_1)=0$ and $c_2(\xi_2)=w(\beta+w)=0$ which follows from (\ref{equation:phi-3}). Therefore $\xi_1\cong\xi_2$ by Proposition~\ref{proposition:isomorphism of sum of two line bundle}, and hence $M(\alpha)=P(\xi_1)\cong P(\xi_2)=P(\gamma^{w}(\gamma^\beta \oplus\mathbb C))\cong P(\gamma^\beta \oplus\mathbb C)=M(\beta)$. This proves (2)$\Rightarrow$(1).

That (3)$\Leftrightarrow$(4) is obvious. If $\phi \colon H^\ast(M(\alpha)) \to H^\ast(M(\beta))$ is any isomorphism, the proof (2)$\Rightarrow$(3) shows that $\phi(\alpha^2) = \beta^2$. Hence by the identity \eqref{eqn : Pontrjagin} the isomorphism $\phi$ preserves the Pontrjagin classes of $M(\alpha)$ and $M(\beta)$.
\end{proof}

By putting all the results together we can conclude the following cohomological rigidity result
for quasitoric manifolds.

\begin{theorem}\label{theorem:final theorem}
Let $M$ and $N$ be $2n$-dimensional quasitoric manifolds whose cohomologies are BQ-algebra of rank
$n$ over $\Z$ with cohomological complexities equal to $1$.
If $H^\ast(M)\cong H^\ast(N)$, then $M$ and $N$ are homeomorphic.
\end{theorem}

\begin{proof}
By Theorem~\ref{theorem:twist number = complexity} and Corollary~\ref{coro:well definedness of twist number}
both $M$ and $N$ are equivalent to one-twist $n$-stage Bott manifolds.
By Theorem~\ref{theorem:rigidity of one-twist Bott tower} those one-twist Bott manifolds are diffeomorphic.
Hence $M$ and $N$ are homeomorphic.
\end{proof}

\section{BQ-algebra over $\Z_{(2)}$} \label{sectoin:BQ-algebra over Z_2}

All the results in previous sections are concerned with BQ-algebras over $\Z$. In this section we remark that these results are stil true for BQ-algebras over the localized ring $\Z_{(2)}$ at $2$.

A BQ-algebra over $\Z$ is defined in Definition~\ref{def:BQ-algebra}. However this definition can be extended to any commutative ring $R$. Namely, a BQ-algebra $S$ of rank $n$ over $R$ is a graded $R$-algebra with generators $x_1, \ldots, x_n$ of degree $2$ such that
\begin{enumerate}
\item $x_k^2=\sum_{i<k}\c{k}{i}x_ix_k$ where $\c{k}{i}\in R$ for $1\le k\le n$, (in particular $x_1^2=0$,) and
\item $\prod_{i=1}^n x_i\ne 0$.
\end{enumerate}
The $R$-complexity of $S$ is the number of $k$'s such that $x_k^2 \neq 0$ in the above condition (1) for all possible choices of generator sets $\{x_1, \ldots, x_n\}$. Note that the cohomology ring $H^\ast(M,R)$ of a quasitoric manifold $M$ is a BQ-algebra over $R$. If $R = \Z$ and $M$ is a Bott manifold, the cohomological complexity of $M$ defined in Section~\ref{section: twist number} is the $\Z$-complexity of $H^\ast(M\colon \Z)$.

In Theorem~\ref{theorem:twist number = complexity} we show that the twist number of a Bott manifold $M$ is equal to the cohomological complexity of $M$. If we examine the proof carefully, the proof is based on arguments whether the coefficients are even or odd. Therefore we can see easily that the same argument works if the integer coefficients are replaced by the localized ring $\Z_{(2)}$ at $2$. Therefore Theorem~\ref{theorem:twist number = complexity} and Corollary~\ref{coro:well definedness of twist number} can be extended as follows.

\begin{theorem} \label{thm:twistnumber_local2}
Let $M$ be a Bott manifold. Then the twist number of $M$ is well-defined and is equal to the $\Z_{(2)}$-complexity of the BQ-algebra $H^\ast(M\colon \Z_{(2)})$. In particular, the $\Z$-complexity of $H^\ast(M\colon \Z)$ is equal to the $\Z_{(2)}$-complexity of $H^\ast(M\colon \Z_{(2)})$.
\end{theorem}

In Theorem~\ref{theorem:BQ-algebra and Bott tower}, it is shown that if $M$ is a quasitoric manifold whose integral cohomology ring is a BQ-algebra over $\Z$, then $M$ is equivalent to a Bott manifold. In its proof, the only place where the property of integral coefficients different from that of rational coefficients is used is where $a_{st}a_{ts} = 2$ implies $a_{st}=\pm1$ and $a_{ts}=\pm2$ right after equation~(\ref{equation:2b1sb1t}).
But this is still true if the coefficient ring is $\mathbb Z_{(2)}$, the integer ring localized at $2$.
Therefore Theorem~\ref{theorem:BQ-algebra and Bott tower} is still true if the coefficient ring is $\mathbb Z_{(2)}$. Therefore Theorem~\ref{theorem:BQ-algebra and Bott tower} can be extended as follows.
\begin{theorem}\label{thm:local2,BQ-algebra vs Bott tower}
Let $M$ be a $2n$-dimensional quasitoric manifold over $P$, and
let $A$ be the characteristic matrix of $M$.
Then the following are equivalent.
\begin{enumerate}
\item $M$ is equivalent to an $n$-stage Bott manifold.
\item $H^\ast(M \colon \Z)$ is a BQ-algebra of rank $n$ over $\Z$.
\item $H^\ast(M:\mathbb Z_{(2)})$ is a BQ-algebra of rank $n$ over $\mathbb Z_{(2)}$.
\item $P$ is combinatorially equivalent to the cube $I^n$ and  $A$ is an
$n\times n$ matirx conjugate
to an upper triangular matrix by a permutation matrix.
\end{enumerate}
\end{theorem}

If we examine the proof of Theorem~\ref{theorem:rigidity of one-twist Bott tower} carefully, we can also see that a similar proof works for the following claim: if $M(\alpha)$ and $M(\beta)$ are one-twist Bott manifols with $H^\ast(M(\alpha)\colon \Z_{(2)}) \cong H^\ast(M(\beta)\colon\Z_{(2)})$, then they are diffeomorphic. So combining this claim together with Theorems~\ref{thm:twistnumber_local2} and \ref{thm:local2,BQ-algebra vs Bott tower}, we can have the following theorem.
\begin{theorem}
Let $M$ and $N$ be $2n$-dimensional quasitoric manifolds whose cohomologies are BQ-algebra of rank $n$ over $\Z_{(2)}$ with $\Z_{(2)}$-complexities less than or equal to $1$. If $H^\ast(M\colon \Z_{(2)}) \cong H^\ast(N \colon \Z_{(2)})$, then $M$ and $N$ are homeomorphic.
\end{theorem}

In the proof of the cohomological rigidity of three-stage Bott manifolds in \cite{CMS}, Wall and Juppe's results on classification of simply connected $6$-dimensional manifolds is used essentially. However, recently, a different but direct proof of the cohomological rigidity of three-stage Bott manifolds is found, and a similar proof also works for the claim that two three-stage Bott manifolds with isomorphic $\Z_{(2)}$-cohomology rings are diffeomorphic. Therefore the same argument as above we have the following theorem.

\begin{theorem}\label{thm:3-stage_local2_version}
Let $M$ and $N$ be $6$-dimensional quasitoric manifolds whose $\Z_{(2)}$-cohomology rings are BQ-algebras over $\Z_{(2)}$. If $H^\ast(M\colon \Z_{(2)}) \cong H^\ast(N \colon \Z_{(2)})$ as graded rings, then $M$ and $N$ are diffeomorphic.
\end{theorem}

More precise argument for Theorem~\ref{thm:3-stage_local2_version} will be shown elsewhere.

\bigskip
\bibliographystyle{amsplain}

\end{document}